\documentclass[11pt,draft]{amsart}
\usepackage{times}
\usepackage{dsfont}

\author{L. L. de Lima}
\address{Departamento de Matem\'atica\\ Universidade Federal do Cear\'a\\ Brazil}
\author{P. Piccione}
\address{Departamento de Matem\'atica\\ Universidade de S\~ao Paulo\\ Brazil}
\author{M. Zedda}
\address{Universit\`a degli Studi di Cagliari\\ Italy}
\curraddr{Departamento de Matem\'atica\\ Universidade de S\~ao Paulo\\ Brazil}
\title[Uniqueness of solutions of the Yamabe problem]{A note on the uniqueness of solutions\\ for the Yamabe problem}
\date{Revised version of June 1st, 2011}
\subjclass[2000]{53C25, 58E11}
\thanks{The first author is partially sponsored by CNPq and Funcap, Brazil. The second author is partially sponsored
by CNPq and Fapesp, Brazil. The third author is supported by \emph{RAS}
through a grant financed with the ``Sardinia PO FSE 2007-2013'' funds and
provided according to the L.R.\ 7/2007.}
\begin{document}
\swapnumbers
\theoremstyle{plain}\newtheorem*{teon}{Theorem}
\theoremstyle{definition}\newtheorem*{defin*}{Definition}
\theoremstyle{plain}\newtheorem{teo}{Theorem}
\theoremstyle{plain}\newtheorem{prop}[teo]{Proposition}
\theoremstyle{plain}\newtheorem*{prop_n}{Proposition}
\theoremstyle{plain}\newtheorem{lem}[teo]{Lemma}
\theoremstyle{plain}\newtheorem*{lem_n}{Lemma}
\theoremstyle{plain}\newtheorem{cor}[teo]{Corollary}
\theoremstyle{definition}\newtheorem{defin}[teo]{Definition}
\theoremstyle{remark}\newtheorem{rem}[teo]{Remark}
\theoremstyle{plain} \newtheorem{assum}[teo]{Assumption}
\swapnumbers
\theoremstyle{definition}\newtheorem{example}{Example}[section]
\theoremstyle{plain}\newtheorem*{acknowledgement}{Acknowledgements}
\theoremstyle{definition}\newtheorem*{notation}{Notation}

\maketitle
\begin{abstract}
Using recent results on the compactness of the space of solutions of the Yamabe problem we show that in conformal classes of metrics near the class of a nondegenerate solution which is unique (up to scaling) the Yamabe problem has a unique solution as well. This provides examples of a local extension, in the space of conformal classes, of a well-known uniqueness criterion due to Obata.
\end{abstract}

\begin{section}{Introduction}
A celebrated result obtained by a combined effort of Yamabe \cite{Yam60}, Trudinger \cite{Tru68}, Aubin \cite{Aub76} and Schoen \cite{Schoen84} gives the existence of a constant scalar curvature metric in every
conformal class of Riemannian metrics on a compact manifold $M$.
Multiplicity results for solutions of the Yamabe problem have been proved in a variety of cases,
see for instance \cite{dLPZ, Petean08}.

Recently, remarkable results on the non compactness of the space of solutions for the Yamabe problem have been
obtained, first by Brendle, see \cite{Brendle}, and subsequently by Brendle and Marques, see  \cite{BreMar09}. A compactness theorem for $n\leq 24$ was proved in \cite{KhuMarSch} by Khuri,
Marques and Schoen.
Manuscript \cite{BreMar10} is an interesting survey on the compactness and the non compactness issues for the Yamabe problem. 

In the present paper we address the question
of \emph{uniqueness} of solutions of the Yamabe problem in conformal classes. A well-known result by Obata \cite{Oba72} says
that an Einstein metric, which happens to be a critical point of the Hilbert--Einstein functional in the manifold
of all Riemannian metrics with total volume $1$, is the unique  constant scalar curvature metric
in its conformal class up to scaling, provided it is not conformal to a round metric on the sphere. Such metrics are nondegenerate global minima of the Hilbert--Einstein
functional. The structure of the set of constant scalar curvature metrics in conformal classes near
a nondegenerate constant scalar curvature metric has been studied in the smooth case by Koiso \cite{Koiso79},
see also \cite[Ch.~4]{Besse}, using the language of ILH (inverse limit Hilbert) manifolds.
In this paper we will establish that, near a nondegenerate solution of the Yamabe problem, the set of unit volume and
constant scalar curvature metrics of class $\mathcal C^{k,\alpha}$, $k\ge2$, on a given compact manifold $M$,
with $m=\mathrm{dim}(M)\ge3$, is a smooth embedded submanifold of
the set of metrics, which is (strongly) transversal to the conformal classes (Proposition~\ref{thm:mankappaconstant}). This  transversality property immediately yields a local uniqueness result for the Yamabe problem in conformal classes near to that of a nondegenerate solution (Corollary~\ref{thm:localuniqueness}).
The proof is based on an abstract formulation of the inverse mapping theorem in vector bundles,
discussed in Section~\ref{sec:fiberwiseIFT}, whose proof only involves standard infinite dimensional
differential geometry techniques, thus completely avoiding the ILH formalism.
Moreover, the local uniqueness result, combined with the compactness results of \cite{KhuMarSch}, \cite{LiZhang, LiZhang2} and
\cite{Mar}, gives a global uniqueness result for solutions of the Yamabe problem in conformal classes near that of a  nondegenerate constant scalar curvature metric which is the unique unit volume
solution in its conformal class (Theorem~\ref{thm:main}). This applies in particular to  certain families of Einstein metrics appearing in Obata's uniqueness result mentioned above. Examples include any space form $M$ (other than the round sphere) satisfying either $m\leq 7$ or $m\leq 24$ and $M$ is spin, or any homogeneous Einstein manifold which is not locally conformally flat.
Thus our result provides examples of a local extension, in the space of conformal classes, of Obata's uniqueness criterion. Observe also that our global uniqueness result generalizes a well known result
by B\"ohm, Wang and Ziller that near a unit volume Einstein metric which is not conformally equivalent to the round sphere all unit volume scalar curvature metrics must be Yamabe metrics, i.e., global
minima of the Hilbert--Einstein functionals, see \cite[Theorem~5.1]{BohWanZil}.
\end{section}

\begin{section}{A fiberwise inverse function theorem}
\label{sec:fiberwiseIFT}
Let us first state an abstract formulation of an inverse mapping theorem which is suited to our purposes.
Let us introduce the following notations, terminology and conventions. Given a Banach manifold
$\mathcal M$, by a smooth Banach vector bundle $\mathcal E$ on $\mathcal M$ with typical fiber $E_0$
we always mean a vector bundle whose transition maps are smooth as maps from open subsets $U\subset\mathcal M$ to
$\mathrm{GL}(E_0)$. The zero section of a vector bundle will be denoted by $\mathbf 0$;
given $x_0\in\mathcal M$, the zero of the fiber $\mathcal E_{x_0}$ is denoted by $\mathbf 0_{x_0}$.
If $s:\mathcal M\to\mathcal E$ is a smooth section of the vector bundle $\mathcal E$ and $s(x_0)=\mathbf 0_{x_0}$,
then the \emph{vertical derivative} $\mathrm d^{\mathrm{ver}}s(x_0):T_{x_0}\mathcal M\to\mathcal E_{x_0}$ is the linear map
defined as the composition $P_{\mathrm{ver}}\circ\mathrm ds(x_0)$, where
$P_{\mathrm{ver}}:T_{\mathbf 0_{x_0}}\mathcal E\to\mathcal E_{x_0}$ is the projection relative to the decomposition
$T_{\mathbf 0_{x_0}}\mathcal E\cong T_{\mathbf 0_{x_0}}\mathbf 0\oplus T_{\mathbf 0_{x_0}}\mathcal E_{x_0}\cong T_{\mathbf 0_{x_0}}\mathbf 0\oplus\mathcal E_{x_0}$.
By a smooth distribution on $\mathcal M$ we mean a smooth vector subbundle of $T\mathcal M$; a submanifold
$\mathcal L\subset\mathcal M$ is said to be \emph{strongly transversal} to a distribution $\mathcal D\subset T\mathcal M$ if
at every point $x\in\mathcal L$, $T_x\mathcal L$ is a closed complement of $\mathcal D_x$, i.e.,
if $T_x\mathcal L\cap\mathcal D_x=\{0\}$ and $T_x\mathcal L+\mathcal D_x=T_x\mathcal M$.
\begin{prop}\label{thm:abstractinversemapping}
Let $\mathcal M$ be a Banach manifold, $\mathcal D\subset T\mathcal M$ be a smooth distribution on $\mathcal M$,
$\mathcal E\to\mathcal M$ a smooth vector bundle and $\mathfrak i:\mathcal E\to\mathcal D^*$ an injective smooth vector bundle morphism.
Assume that $F:\mathcal M\to\mathds R$ is a smooth function and that $s$ is a smooth section of $\mathcal E$ satisfying:
\begin{equation}\label{eq:reliDf}
\phantom{,\quad\forall\,x\in\mathcal M.}\mathfrak i\big(s(x)\big)=\mathrm dF(x)\big\vert_{\mathcal D_x},\quad\forall\,x\in\mathcal M.
\end{equation}
Assume that $\mathrm dF(x_0)\big\vert_{\mathcal D_{x_0}}=0$  (or equivalently, that $s(x_0)=0$)
at some $x_0\in\mathcal M$, and that
\[\mathrm d^{\mathrm{ver}}s(x_0)\big\vert_{\mathcal D_{x_0}}:\mathcal D_{x_0}\longrightarrow\mathcal E_{x_0}\]
is an isomorphism of Banach spaces.

Then, there exists an open subset $U\subset\mathcal M$ containing $x_0$ such that the set:
\[\mathcal L=\big\{x\in U:\mathrm dF(x)\big\vert_{\mathcal D_x}=0\big\}.\]
is a smooth embedded submanifold of $\mathcal M$ which is strongly transversal to $\mathcal D$.
\end{prop}
\begin{proof}
The assumption that $\mathrm d^{\mathrm{ver}}s(x_0)$ is an isomorphism from $\mathcal D_{x_0}$ to $\mathcal E_{x_0}$
implies that $s$ is transversal to the zero section of $\mathcal E$ at $x_0$.
Observe that the composition of $\mathrm ds(x_0):T_{x_0}\mathcal M\to T_{\mathbf 0_{x_0}}\mathcal E$ with the projection
$T_{\mathbf 0_{x_0}}\mathcal E\to T_{\mathbf 0_{x_0}}\mathcal E/T_{x_0}\mathbf 0\cong\mathcal E_{x_0}$ is identified with
the vertical derivative $\mathrm d^{\mathrm{ver}}s(x_0)$, which is therefore surjective. Moreover, the
space $\mathcal D_{x_0}$ is a closed complement of $\mathrm{Ker}\big(\mathrm ds(x_0)\big)$, i.e., $s$ is strongly transversal to $\mathbf 0$ at $x_0$.
Since $\mathfrak i$ is injective, from \eqref{eq:reliDf} it follows easily that the set $\mathcal L$ can also
be described as:
\[\mathcal L=U\cap s^{-1}(\mathbf 0);\]
it follows that for $U$ sufficiently small, $\mathcal L$ is a smooth embedded submanifold of $\mathcal M$, and that
for $x\in\mathcal L$, $T_x\mathcal L=\mathrm{Ker}\big(\mathrm ds(x)\big)$. Note that for $x\in\mathcal L$ near $x_0$, $\mathcal D_x$
is a closed complement of $\mathrm{Ker}\big(\mathrm ds(x)\big)$, i.e., $\mathcal L$ is strongly transversal to $\mathcal D$.
\end{proof}
\end{section}
\begin{section}{The manifold of unit volume metrics and conformal classes}
We will now discuss an application of Proposition~\ref{thm:abstractinversemapping} in the context of constant
scalar curvature metrics.
Let  $M$ be a compact manifold, $m=\mathrm{dim}(M)\ge3$, and given $k\ge2$, $\alpha\in\left]0,1\right]$, denote by
$\mathrm{Met}^{k,\alpha}(M)$ the set of Riemannian metrics of class $\mathcal C^{k,\alpha}$ on $M$, i.e., the open
subset of the Banach space $\Gamma^{k,\alpha}(\vee^2TM^*)$ consisting of all $\mathcal C^{k,\alpha}$-sections of the vector bundle of $\vee^2(TM^*)$
consisting of symmetric (0,2)-tensors on $M$ that are pointwise positive definite.
For $g\in\mathrm{Met}^{k,\alpha}(M)$, we will denote by $\nu_g$ the volume form (or density, if $M$ is not orientable) of $g$, by $\mathrm{Ric}_g$
the Ricci curvature of $g$, and by $\kappa_g$
its scalar curvature function, which is a function of class $\mathcal C^{k-2,\alpha}$ on $M$.
The \emph{volume function} $\mathcal V$ on $\mathrm{Met}^{k,\alpha}(M)$ is defined by:
\[\mathcal V(g)=\int_M\nu_g.\]
Observe that $\mathcal V$ is smooth, and its differential is given by:
\begin{equation}\label{eq:difvolume}
\mathrm d\mathcal V(g)h=\tfrac12\int_M\mathrm{tr}_g(h)\,\nu_g,
\end{equation}
for all $h\in\Gamma^{k,\alpha}(\vee^2TM^*)$.
Set:
\begin{equation}\label{eq:mathcalMM}
\mathcal M=\big\{g\in\mathrm{Met}^{k,\alpha}(M):\mathcal V(g)=1\big\};
\end{equation}
this is a smooth embedded submanifold of $\mathrm{Met}^{k,\alpha}(M)$. For $g\in\mathcal M$,
the tangent space $T_g\mathcal M$ is given by:
\[T_g\mathcal M=\Big\{h\in\Gamma^{k,\alpha}\big(\vee^2(TM^*)\big):\smallint_M\mathrm{tr}_g(h)\,\nu_g=0\Big\}.\]
Consider the Hilbert--Einstein functional $F:\mathcal M\longrightarrow\mathds R$,
defined by:
\begin{equation}\label{eq:HEfunctional}
F(g)=\int_M\kappa_{g}\,\nu_{g};
\end{equation}
this is a smooth function on $\mathcal M$, whose critical points are the \emph{Einstein metrics} of volume $1$ on $M$. Finally, for $g\in\mathcal M$, let $\mathcal D_g$ be the closed subspace of $T_g\mathcal M$ given by:
\[\mathcal D_g=\big\{\varphi\cdot g:\varphi\in\mathcal C^{k,\alpha}(M,\mathds R):\smallint_M\varphi\,\nu_g=0\big\}.\]
\begin{lem}\label{thm:lemdistribution}
As to the above setting, the following hold:
\begin{itemize}
\item[(a)] $\mathcal D=\bigcup_{g\in\mathcal M}\mathcal D_g$ is a smooth integrable distribution on $\mathcal M$;
\item[(b)] for $g\in\mathcal M$, the maximal connected integral submanifold of $\mathcal D$ through $g$ is the subset $D_g\subset\mathcal M$ given by:
\[D_g=\big\{\phi\cdot g:\phi\in\mathcal C^{k,\alpha}(M,\mathds R):\phi>0,\ \int_M\phi^{\frac m2}\nu_g=1\big\}.\]
\end{itemize}
\end{lem}
\begin{proof}
Statement (a) is proven by showing that $\mathcal D$ is the image of a smooth vector bundle morphism with
\emph{constant rank}.\footnote{Given vector bundles $\mathcal E$ and $\mathcal F$ over a manifold $M$, with typical fibers
$E_0$ and $F_0$ respectively, a vector bundle morphism $\phi:\mathcal E\to\mathcal F$ is said to have constant
rank if, identifying the fiber $\mathcal E_x\cong E_0$, $\mathcal F_x\cong F_0$, $x\in U\subset M$, via local trivializations
of $\mathcal E$ and $\mathcal F$, there exists a closed subspace $G\subset F_0$ such that $\phi_x(E_0)\oplus G=F_0$ for all $x\in U$.
Such condition does not depend on the choice of trivializations. If $\phi$ has constant rank, then $\phi(\mathcal E)$ is a subbundle
of $\mathcal F$.
\begin{lem_n}
If $\phi:\mathcal E\to\mathcal F$ has a left inverse, then it has fiberwise closed image and constant rank.
\end{lem_n}
\begin{proof}
A bounded linear operator $T:E_0\to F_0$ between Banach spaces that admits a left inverse $S:F_0\to E_0$ has closed image, and $\mathrm{Ker}(S)$ is a closed complement for $\mathrm{Im}(T)$ (trivial).
Given trivializations of $\mathcal E$ and $\mathcal F$ over an open subset $U\subset M$, write
$\phi$ and its left inverse $\psi$ as maps $U\ni x\mapsto\phi_x\in\mathrm{Lin}(E_0,F_0)$ and
$U\ni x\mapsto\psi_x\in\mathrm{Lin}(F_0,E_0)$. Now, given $x_0\in U$, for all $x\in U$ the map
$(\psi_{x_0}\circ\phi_x)^{-1}\circ\psi_{x_0}$ is a left inverse of $\phi_x$. It follows that:
\[\mathrm{Ker}\big((\psi_{x_0}\circ\phi_x)^{-1}\circ\psi_{x_0}\big)=\mathrm{Ker}(\psi_{x_0})\]
is a fixed closed complement of $\mathrm{Im}(\phi_x)$ for all $x\in U$, i.e., $\phi$ has constant rank.
\end{proof}
}
Consider the trivial vector bundles (over $\mathcal M$) $E_1=\mathcal M\times\mathcal C^{k,\alpha}(M,\mathds R)$, $E_2=\mathcal M\times\mathds R$,
and the smooth vector bundle morphism $f_1:E_1\to E_2$ defined by \[f_1(g,\varphi)=\big(g,\smallint_M\varphi\,\nu_g\big).\]
Clearly, $f_1$ is surjective, and its kernel
\[E_3=\bigcup_{g\in\mathcal M}\{g\}\times\big\{\varphi\in\mathcal C^{k,\alpha}(M,\mathds R):\smallint_M\varphi\,\nu_g=0\big\},\]
which is closed and complemented, is a smooth vector subbundle of $E_1$. Now, the smooth bundle morphism $f_2:E_3\to T\mathcal M$
defined by $(g,\varphi)\mapsto(g,\varphi\cdot g)$ has image $\mathcal D$, and it has constant rank.
Namely, the fiber bundle morphism $T\mathcal M\ni (g,\overline g)\mapsto\big(g,\frac1m\mathrm{tr}_g(\overline g)\big)\in E_3$ is a left inverse of $f_2$. This proves (a).
For (b), observe that $\phi^\frac m2\nu_g$ is the volume form of the metric $\phi\cdot g$.
For $g\in\mathcal M$, consider the smooth map $\mathcal V_g:\mathcal C^{k,\alpha}(M,\mathds R)\ni\phi\mapsto\int_M\phi^\frac m2\nu_g\in\mathds R$,
whose derivative is
\begin{equation}\label{eq:derVg}
\mathrm d\mathcal V_g(\phi)\varphi=\frac m2\int_M\phi^{\frac m2-1}\varphi\,\nu_g.
\end{equation}
It is easily checked that $\mathcal V_g$ has no critical values, so that the
set $\mathcal V_g^{-1}(1)$ is a smooth embedded submanifold of $\mathcal C^{k,\alpha}(M,\mathds R)$.
Now, for $g\in\mathcal M$, the map $I:\mathcal V_g^{-1}(1)\ni\phi\mapsto\phi\cdot g\in\mathcal M$ is smooth, and a smooth left-inverse is
given by the map $\overline g\mapsto\frac1m\mathrm{tr}_g(\overline g)$. It follows that $I$ is an embedding, and
its image is $D_g$, which is therefore a smooth embedded submanifold of $\mathcal M$.
Using \eqref{eq:derVg} at $\phi=1$, one proves easily $T_gD_g=\mathcal D_g$, i.e., $D_g$ is an integral submanifold of $\mathcal D$. It is easy to see that $D_g$ is connected (it is homeomorphic to the set
of positive functions $\phi\in\mathcal C^{k,\alpha}$ satisfying $\int_M\phi^\frac m2\,\nu_g=1$),
and that $\mathcal M$ is the disjoint union of all $D_g$'s. It follows\footnote{We are using the following elementary result:
\begin{lem_n}Let $\mathcal M$ be a manifold and $\mathcal D\subset T\mathcal M$ an integrable distribution.
Assume that $\mathcal M$ is the disjoint union $\bigcup_{i\in I}S_i$
of connected integral submanifolds of $\mathcal D$. Then, every $S_i$ is a maximal connected integral submanifold of $\mathcal D$.\end{lem_n}}
that every $D_g$ is a maximal connected
integral submanifold of $\mathcal D$.
\end{proof}
The \emph{$\mathcal C^{k,\alpha}$-conformal class} of $g$ is the submanifold of $\mathrm{Met}^{k,\alpha}(M)$
\[\big[g\big]_{k,\alpha}=\big\{\phi\cdot g:\phi\in\mathcal C^{k,\alpha}(M,\mathds R):\phi>0\big\};\]
clearly, $D_g=\mathcal M\cap\big[g\big]_{k,\alpha}$. From the proof of Lemma~\ref{thm:lemdistribution}, part (b),
 one deduces easily that $\mathcal M$ and $\big[g\big]_{k,\alpha}$ are transversal.\smallskip

Let us recall the first and the second variational formula for the function $F$ restricted to a conformal
class $D_g$, see \cite{Koiso78, Schoen87}
for details. We identify the tangent space $T_gD_g$ with the space of
 functions $f\in\mathcal C^{k,\alpha}(M,\mathds R)$  with $\int_Mf\,\nu_g=0$. Given one such $f$, the first derivative
 of $F\big\vert_{D_g}$ in the direction $f$ is:
\begin{equation}\label{eq:firstvar}
\mathrm dF(g)(f)=\frac{m-2}2\int_Mf\kappa_g\,\nu_g;
\end{equation}
if $g$ has constant scalar curvature, then:
\begin{equation}\label{eq:secondvar}
\mathrm d^2F(g)\big(f,f\big)=\frac{m-2}2\int_M\big((m-1)\Delta_gf-\kappa_gf\big)f\,\nu_g,
\end{equation}
where $\Delta_g$ is the Laplace--Beltrami operator of the metric $g$.

\end{section}
\begin{section}{The manifold of unit volume constant scalar curvature metrics}
Given $g\in\mathcal M$ with constant scalar curvature $\kappa_g$ in $M$, we say that $g$ is \emph{nondegenerate}
if either $\kappa_g=0$ or if $\frac{\kappa_g}{m-1}$ is not an eigenvalue of the Laplace--Beltrami operator
$\Delta_g$ of $g$.
\begin{prop}\label{thm:mankappaconstant}
Let $g_*\in\mathcal M$ be a nondegenerate constant scalar curvature metric. Then, there exists an
open neighborhood $U$ of $g_*$ in $\mathcal M$ such that the set:
\begin{equation}\label{eq:manunitvolconstsccurv}
\big\{g\in U: \kappa_g\ \text{is constant}\big\}
\end{equation}
is a smooth embedded submanifold of $\mathcal M$ which is strongly transversal to the conformal classes.
\end{prop}
\begin{proof}
This is a direct application of Proposition~\ref{thm:abstractinversemapping} to the following setup.
The manifold $\mathcal M$ is defined in \eqref{eq:mathcalMM} and the distribution $\mathcal D$ in
Lemma~\ref{thm:lemdistribution}. The vector bundle $\mathcal E$ has fibers:
\[\mathcal E_g=\big\{f\in\mathcal C^{k-2,\alpha}(M,\mathds R):\smallint_Mf\,\nu_g=0\big\},\quad g\in\mathcal M,\]
and the vector bundle morphism $\mathfrak i:\mathcal E\to\mathcal D^*$ is induced by the $L^2$-pairing,
i.e.,
\[\mathfrak i_g(f_1)f_2=\int_Mf_1f_2\,\nu_g,\quad g\in\mathcal M,\ f_1\in\mathcal E_g,\ f_2\in\mathcal D_g.\]
The function $F:\mathcal M\to\mathds R$ is the Hilbert--Einstein functional \eqref{eq:HEfunctional}, and
the section $s:\mathcal M\to\mathcal E$ is given by:
\begin{align*}
s(g)&=-\mathrm{tr}_g\big(\mathrm{Ric}_g-\tfrac12g\,\kappa_g\big)+\int_M\big(\mathrm{Ric}_g-\tfrac12g\,\kappa_g\big)\nu_g\\
&=\big(\tfrac m2-1\big)\kappa_g-\big(\tfrac m2-1\big)\int_M\kappa_g\,\nu_g.
\end{align*}
Note that $s(g)=0$ if and only if $\kappa_g$ is constant. Given one such $g$, the composition
$\mathfrak i_g\circ\mathrm d^{\mathrm{ver}}s(g)\big\vert_{\mathcal D_g}:\mathcal D_g\to\mathcal D_g^*$
is identified with the second variation $\mathrm d^2(F\vert_{D_g})(g)$ of the restriction of $F$ to
the integral submanifold $D_g$.
Using \eqref{eq:firstvar} one checks easily that equality \eqref{eq:reliDf} in Proposition~\ref{thm:abstractinversemapping} holds. Moreover,
using formula \eqref{eq:secondvar}, the map $\mathrm d^{\mathrm{ver}}s(g_*)\big\vert_{\mathcal D_{g_*}}:\mathcal D_{g_*}\to\mathcal E_{g_*}$ is identified with the linear operator \[f\longmapsto\frac{m-2}2\big((m-1)\Delta_{g_*}f-\kappa_{g_*}f\big),\]
defined on the space of functions $f\in\mathcal C^{k,\alpha}(M,\mathds R)$ with $\int_Mf\,\nu_{g_*}=0$ and taking values
in the space of functions $f\in\mathcal C^{k-2,\alpha}(M,\mathds R)$ with $\int_Mf\,\nu_{g_*}=0$. Observe that such operator
preserves the space of functions with zero integral.
The assumption of nondegeneray for $g_*$ is equivalent to the fact that such operator is an isomorphism, and Proposition~\ref{thm:abstractinversemapping}
can be applied to obtain the conclusion.
\end{proof}

\begin{cor}\label{thm:localuniqueness}
Let $g_*\in\mathcal M$ be a nondegenerate constant scalar curvature metric. Then, there exists an
open neighborhood $U$ of $g_*$ in $\mathcal M$ such that every conformal class of metrics in $M$ contains at most
one unit volume constant curvature metric in $U$.
\end{cor}
\begin{proof}
Strong transversality of the manifold \eqref{eq:manunitvolconstsccurv} with the conformal classes implies
local uniqueness of intersections.
\end{proof}
We are ready for our main result, which gives a \emph{global} uniqueness result for
the Yamabe problem.
\begin{teo}\label{thm:main}
Assume that $g_*\in\mathcal M$ is a nondegenerate constant scalar curvature metric on $M$, which is the unique unit volume
constant scalar curvature metric in its conformal class. Assume in addition that either one of the following conditions is satisfied:
\begin{itemize}
\item[(a)] $\mathrm{dim}(M)\le7$;
\item[(b)] the Weyl tensor $W_{g_*}$ of $g_*$ satisfies
\begin{equation}\label{eq:condWeyltensor}
\phantom{,\quad\forall\,p\in M}\big\vert W_{g_*}(p)\big\vert+\big\vert\nabla W_{g_*}(p)\big\vert>0,\quad\forall\,p\in M;
\end{equation}
\item[(c)] $\mathrm{dim}(M)\le24$ and $M$ is spin.
\end{itemize}
Then, for $g$ sufficiently $\mathcal C^{2,\alpha}$-close to $g_*$, there is a unique unit volume constant scalar curvature metric
in the conformal class of $g$.
\end{teo}
\begin{proof}
Either one of conditions (a), (b) or (c) guarantees that, given a sequence $g_n$ of Riemannian metrics on $M$ with $\lim\limits_{n\to\infty}g_n=g_*$,
then the set unit volume constant scalar curvature metrics that are conformal to some $g_n$ is compact. This follows from the arguments in
\cite{KhuMarSch, LiZhang, Mar}, see in particular \cite[Lemma10.1]{KhuMarSch}. Note that condition \eqref{eq:condWeyltensor} is open (in the
$\mathcal C^1$-topology), hence it holds for every $g_n$ with $n$ sufficiently large if it holds for $g_*$.
By contradiction, assume the existence of a sequence $g_n$ of unit volume constant scalar curvature metrics on $M$, with $\lim\limits_{n\to\infty}g_n=g_*$,
and such that the conformal class of $g_n$ contains a distinct unit volume constant scalar curvature metric $h_n\ne g_n$ for all $n$.
By compactness, up to subsequences there exists the limit $\lim\limits_{n\to\infty}h_n=h_*$; by continuity, $h_*$ is a unit volume
constant scalar curvature metric in the conformal class of $g_*$. By uniqueness in the conformal class of $g_*$, it must be $h_*=g_*$, but this contradicts the local uniqueness
result of Corollary~\ref{thm:localuniqueness}, which concludes the proof.
\end{proof}
\end{section}

\end{document}